\documentclass[12pt,leqno,a4paper]{amsart}
\usepackage{amssymb,enumerate}
\newcommand{\FF}{{\mathbb{F}}}

\newcommand{\fA}{{\mathfrak{A}}}
\newcommand{\fS}{{\mathfrak{S}}}

\newcommand{\bC}{{\mathbf{C}}}
\newcommand{\bG}{{\mathbf{G}}}
\newcommand{\bL}{{\mathbf{L}}}
\newcommand{\bN}{{\mathbf{N}}}
\newcommand{\bT}{{\mathbf{T}}}
\newcommand{\bZ}{{\mathbf{Z}}}

\newcommand{\cE}{{\mathcal{E}}}
\newcommand{\cF}{{\mathcal{F}}}

\newcommand{\Aut}{{\operatorname{Aut}}}
\newcommand{\IBr}{{\operatorname{IBr}}}
\newcommand{\Inn}{{\operatorname{Inn}}}
\newcommand{\Irr}{{\operatorname{Irr}}}
\newcommand{\Out}{{\operatorname{Out}}}
\newcommand{\Syl}{{\operatorname{Syl}}}
\newcommand{\Ht}{{\operatorname{ht}}}

\newcommand{\GL}{{\operatorname{GL}}}
\newcommand{\PGL}{{\operatorname{PGL}}}
\newcommand{\PSL}{{\operatorname{L}}}
\newcommand{\SL}{{\operatorname{SL}}}
\newcommand{\SU}{{\operatorname{SU}}}
\newcommand{\PSU}{{\operatorname{U}}}
\newcommand{\Sp}{{\operatorname{Sp}}}

\newcommand\RtTtG{{\operatorname{R}_\tbT^\tbG}}

\newcommand{\tbG}{{\tilde\bG}}
\newcommand{\tbT}{{\tilde\bT}}
\newcommand{\tG}{{\tilde G}}
\newcommand{\tT}{{\tilde T}}

\newcommand{\tw}[1]{{}^#1\!}

\let\vhi=\varphi
\let\la=\lambda

\def\nor{\triangleleft\,}
\def\cent#1#2{{\bf C}_{#1}(#2)}
\def\oh#1#2{{\bf O}_{#1}(#2)}
\def\sbs{\subseteq}

\newtheorem{thm}{Theorem}[section]
\newtheorem{lem}[thm]{Lemma}

\newtheorem{cor}[thm]{Corollary}
\newtheorem{prop}[thm]{Proposition}

\newtheorem*{conjA}{Conjecture A}
\newtheorem*{conjB}{Conjecture B}
\newtheorem*{thmC}{Theorem C}

\theoremstyle{definition}
\newtheorem{exmp}[thm]{Example}

\theoremstyle{remark}
\newtheorem{rem}[thm]{Remark}

\begin{document}

\title[Defects and decomposition numbers]{On defects of characters and decomposition numbers}

\date{\today}

\author{Gunter Malle}
\address{FB Mathematik, TU Kaiserslautern, Postfach 3049,
        67653 Kaisers\-lautern, Germany.}
\email{malle@mathematik.uni-kl.de}
\author{Gabriel Navarro}
\address{Department of Mathematics, Universitat de Val\`encia, 46100 Burjassot,
        Spain.}
\email{gabriel@uv.es}
\author{Benjamin Sambale}
\address{FB Mathematik, TU Kaiserslautern, Postfach 3049,
        67653 Kaisers\-lautern, Germany.}
\email{sambale@mathematik.uni-kl.de}

\begin{abstract}
We propose upper bounds for the number of modular constituents of the
restriction modulo $p$ of a complex irreducible character of a finite group,
and for its decomposition numbers, in certain cases.
\end{abstract}

\thanks{The first author gratefully acknowledges financial support
by ERC Advanced Grant 291512. The second author is partially supported by the
Spanish Ministerio de Educaci\'on y Ciencia Proyectos  MTM2016-76196-P  and
Prometeo II/Generalitat Valenciana. The third author thanks the German Research
Foundation SA 2864/1-1 and Daimler Benz Foundation 32-08/13.
This work begun when the first and second
authors were visiting the CIB in Lausanne. We would like to thank the institute
for its hospitality.}



\maketitle

\section{Introduction}   \label{sec:intro}
Let $G$ be a finite group and let $p$ be a prime. In his fundamental paper
\cite{Br},
Richard Brauer studied the irreducible complex characters $\chi \in \Irr(G)$
such that $\chi(1)_p=|G|_p/p$, where here $n_p$ denotes the largest power
of $p$ dividing the integer $n$. This was the birth of what later became the
cyclic defect theory, developed by E.~C.~Dade \cite{Dade}, building upon work
of J.~A.~Green  and J.~G.~Thompson on vertices and sources. Of
course, Brauer and Nesbitt had studied before the \emph{defect zero}
characters (those $\chi \in \Irr(G)$ with $\chi(1)_p=|G|_p$) proving that
they lift irreducible modular characters in characteristic $p$.
A constant in Brauer's work was to analyse the decomposition
of the complex irreducible characters $\chi$ into modular characters:
$$\chi^0=\sum_{\vhi \in \IBr(G)} d_{\chi \vhi} \vhi ,$$
where here $\chi^0$ is the restriction of $\chi$ to the elements of $G$ of
order prime to $p$, and where we have chosen a set $\IBr(G)$ of
irreducible $p$-Brauer characters of $G$. To better understand the
\emph{decomposition numbers} $d_{\chi \vhi}$ remains one of the challenges
in Representation Theory.
\medskip

If $\chi \in \Irr(G)$ let us write
$\IBr(\chi^0)=\{\vhi \in \IBr(G) \mid d_{\chi\vhi}\ne 0\}$, and
recall that the \emph{defect} of $\chi$ is the integer $d_\chi$ with
$$p^{d_\chi}\chi(1)_p=|G|_p\, .$$

For $\chi \in \Irr(G)$ of defect one Brauer proved in \cite{Br} that all
decomposition numbers $d_{\chi \vhi}$ are less than or equal to 1, and
implicitly, that there are less than $p$ characters $\vhi \in \IBr(G)$
occurring with multiplicity $d_{\chi \vhi} \ne 0$.

\medskip
In order to gain insight into  decomposition numbers in general, it is natural
from this perspective to next study characters of defect two. This step looks
innocent, but it deepens things in such a way that at present we can only guess
what might be happening in general:

\begin{conjA}   \label{conj:main}
 Let $G$ be a finite group, $p$ a prime. Let $\chi\in\Irr(G)$ with
 $|G|_p=p^2\cdot\chi(1)_p$. Then:
 \begin{enumerate}
  \item[\rm(1)] $|\IBr(\chi^0)|\le p^2-1$ and
  \item[\rm(2)] $d_{\chi\vhi}\le p$ for all $\vhi\in\IBr(G)$.
 \end{enumerate}
\end{conjA}

It is remarkable that, as we shall prove  below (Theorem \ref{thm:p2}),
Conjecture A follows from the Alperin--McKay conjecture together with the work
of K.~Erdmann, but only for the prime $p=2$. We do prove below Conjecture A for
$p$-solvable groups (see Theorem~\ref{thm:p-solv}) and for certain classes of
quasi-simple and almost simple groups (see Theorem~\ref{thm:An},
Propositions~\ref{prop:2.An} and~\ref{prop:spor} and
Theorem~\ref{thm:defchar}). On the other hand, Conjecture~A is wide open for
example for groups of Lie type in non-defining characteristic
(see Example~\ref{exmp:GLn}, but also Proposition~\ref{prop:SLn}).
As far as we are aware, no bounds for the number $|\IBr(\chi^0)|$ have been
proposed before, and therefore, we move into unexplored territory here.
\medskip

If Conjecture A is true then both of our bounds are sharp. Of course, the
irreducible characters of degree $p$ in a non-abelian group of order $p^3$ have
decomposition numbers equal to $p$. Also, the irreducible character $\chi$ of
degree $p^2-1$ in the semidirect product of $C_p \times C_p$ with a cyclic
group of order $p^2-1$ satisfies $|\IBr(\chi^0)|=p^2-1$.
\medskip

In view of Brauer's analysis of characters of defect one and our Conjecture A,
it is tempting to guess that whenever $\chi(1)_p=|G|_p/p^3$, then
$d_{\chi \vhi} \le p^2$ and $|\IBr(\chi^0)|\le p^3-1$. While we are not
aware of any $p$-solvable counter-examples, still this is not true in
general: For $p=2$, $G=3.J_3$ has a character $\chi$ such that
$\chi(1)_p=|G|_p/p^3$, having 8 irreducible Brauer constituents. For $p=3$,
$G=Co_3$ has an irreducible character such that $\chi(1)_p=|G|_p/p^3$ with 13
occurring as a decomposition number. Perhaps other bounds are possible.
\medskip

While studying Conjecture A, we came across a remarkable inequality, to which
we have not yet found a counterexample.

\begin{conjB}   \label{conj:lb}
 Let $G$ be a finite group, $p$ a prime, and $\chi\in\Irr(G)$. Then
 $$|\IBr(\chi^0)|\cdot\chi(1)_p\le |G|_p.$$
\end{conjB}

For characters of defect~1, this is the cited result of Brauer, and for
defect~2 it follows from Conjecture~A(1).
\medskip

We shall prove below that Conjecture B is satisfied in symmetric groups and
for simple groups of Lie type in defining characteristic (see
Proposition~\ref{prop:lB-Sn} and Corollary~\ref{cor:defchar-lb}). Our bound
seems exactly the right bound for these classes of groups, and somehow this
makes us  think that Conjecture~B points in the right
direction. On the other hand, we do not know what is happening in other
important classes of groups, like solvable groups or groups of Lie type in
non-defining characteristic, for instance. Whether
Conjecture B is true or false, we believe that the relation between
$|\IBr(\chi^0)|$ and $|G|_p/\chi(1)_p$ is worth exploring.
Of course, if $|G|_p/\chi(1)_p=1$, then $|\IBr(\chi^0)|=1$ by the
Brauer--Nesbitt theorem.

Observe that the stronger inequality $l(B) \cdot\chi(1)_p\le |G|_p$
is not always true. For example, let $G$ be the central product of $n$ copies
of $\SL_2(3)$ where the centres of order $2$ are identified. Then the principal
$2$-block $B$ of $G$ has a normal defect group and satisfies $l(B)=3^n$, but
there is an irreducible character $\chi\in\Irr(B)$ (deflated from the direct
product of $\SL_2(3)$'s) such that $|G|_2/\chi(1)_2=2^{n+1}$.
Other examples are the principal 2-block of $J_3$ or the principal 3-block
of $\SL_6(2)$.
\medskip

So far we have not mentioned blocks, heights, or defect groups. These are, of
course, related to Conjectures A and B.
If $\chi\in\Irr(G)$, then $\chi\in\Irr(B)$ for a unique $p$-block $B$ of $G$,
and $\chi(1)_p=p^{a-d +h}$, where $|G|_p=p^a$, $|D|=p^d$ is the order of a
defect group $D$ for $B$, $d$ is the defect of the block $B$, and $h\ge 0$ is
the height of $\chi$. Hence, the defect $d_\chi$ of $\chi$ is
$$d_\chi=d - h \, .$$
Brauer's famous $k(B)$-conjecture asserts that $k(B):=|\Irr(B)|\le |D|=p^d$.
Since obviously $|\IBr(\chi^0)|\le |\IBr(B)|=:l(B)<|\Irr(B)|$ (if $d>0$), it
follows that for characters of height zero both Conjectures A(1) and B are
implied by Brauer's $k(B)$-conjecture. (In particular, by the Kessar--Malle
solution of one implication of the Height Zero Conjecture \cite{KM13},
Conjectures A(1) and B follow from the $k(B)$-conjecture for characters in
blocks with abelian defect groups.)

A recent conjecture, formulated by Malle and Robinson in \cite{MR16}, is also
related to the present work. Malle and Robinson have proposed that
$l(B)\le p^{s(D)}$, where $s(D)$ is the so called $p$-sectional rank of the
group $D$. Hence Conjecture B follows from the Malle--Robinson conjecture for
those irreducible characters whose height $h$ is such that
\begin{equation}   \label{eq:height}
s(D) + h \le d.
\end{equation}

Finally, we come back to characters of defect 2 for small primes, but from the
perspective of  their relationship with their blocks and defect groups. Here
we prove:

\begin{thmC}
 \begin{enumerate}[\rm(a)]
  \item If $p=2$ then the Alperin--McKay conjecture implies Conjecture~A.
  \item If $p=3$ then Robinson's ordinary weight conjecture implies $l(B)\le10$
   for every block $B$ containing a character $\chi$ as in Conjecture A.
  \end{enumerate}
\end{thmC}

If there is an upper bound for $l(B)$ in Theorem~C for arbitrary primes, we
have not been able to find it. In the situation of Theorem~C for $p$-solvable
groups, we shall prove below that either $|D|\le p^3$ or $|D|=p^4$ and
$p\le 3$, and the possible defect groups are classified. For non $p$-solvable
groups, however, $|D|$ is unbounded, as shown by  $\SL_2(q)$ with $q$ odd
and $p=2$.

\section{$p$-solvable groups}

We start with the proof of Conjecture A for $p$-solvable groups.
In fact, we prove something more general.
Our notation for complex characters follows \cite{Isa}, and for Brauer
characters \cite{N}. If $G$ is a finite group, $N \nor G$, and
$\theta \in \Irr(N)$, then $\Irr(G|\theta)$ is the set of complex
irreducible characters $\chi \in \Irr(G)$ such that the restriction $\chi_N$
contains $\theta$ as a constituent. Notice that if $\psi^G=\chi \in \Irr(G)$
for $\psi$ an irreducible character of the inertia group $T$ of $\theta$, then
$|G:T|_p\psi(1)_p=\chi(1)_p$, and therefore $|G|_p/\chi(1)_p=|T|_p/\psi(1)_p$.
Hence, $d_\chi=d_\psi$.

\medskip
First we collect some rather well-known results.

\begin{lem}   \label{aux}
 Let $p$ be a prime, and let $U$ be a subgroup of $\GL_2(p)$.
 \begin{enumerate}[\rm(a)]
  \item Assume that $U$ has order divisible by $p$. Then either $U$ has a
   normal Sylow $p$-subgroup or $\SL_2(p)\sbs U$.
  \item If $W \sbs \SL_2(p)\sbs U \sbs \GL_2(p)$, then $W$ and $U$ have
   trivial Schur multiplier.
  \item If $W \sbs \SL_2(p)$ is a $p'$-subgroup, then either $W$ is cyclic of
   order a divisor of $p-1$ or $p+1$, $W$ has a normal cyclic subgroup
   of index~2, or $W=\SL_2(3)$, ${\tt SmallGroup}(48,28)$ or $\SL_2(5)$.
  \item Suppose that $G=(C_p \times C_p): U$ in natural action, where
   $\SL_2(p) \sbs U \sbs \GL_2(p)$. If $\chi \in \Irr(G)$,
   and $p$ divides $\chi(1)$, then $\chi(1)=p$.
  \item Suppose that $L$ is an extra-special group of order $p^3$ and exponent
   $p$. Then we have that $\Aut(L)/\Inn(L) \cong \GL_2(p)$. If fact, if
   $Z=\bZ(L)$, $A=\Aut(L)$ and $I=\Inn(L)$, then
   $A=\cent{A} Z:\langle \sigma\rangle$, where $\sigma$ has order $p-1$,
   and $\cent AZ/I=\SL_2(p)$.
 \end{enumerate}
\end{lem}

\begin{proof}
(a)~~This part follows from \cite[8.6.7]{KS04}.

(b)~~Checking $p=3$ directly, we may assume $p\ge 5$.
Suppose by way of contradiction that $U$ has a proper covering group $S$
with $1\ne Z\le \bZ(S) \cap S'$ and $S/Z\cong U$. Let $N\unlhd S$ be the
preimage of $\SL_2(p)$. It is a well-known fact that the Sylow subgroups of
$\SL_2(p)$ are cyclic or quaternion groups. This implies that $\SL_2(p)$
(and any of its subgroups) has trivial
Schur multiplier. In particular, $N$ is not a covering group of
$\SL_2(p)$. Since $\SL_2(p)$ is perfect for $p\ge 5$, we must have
$N=N'Z$ and $Z\nsubseteq N'$. Since $S/N\cong U/\SL_2(p)$ is cyclic, it
follows that $S'=N'$. But this gives the contradiction $Z\nsubseteq S'$.
\smallskip

(c)~~We follow the well-known classification of the subgroups of $\PSL_2(p)$.
Notice that $S$ has a unique involution, so the 2-subgroups of $S$ are cyclic
or generalised quaternion. Set $Z:=\bZ(S)$. Now, if $WZ/Z$ is cyclic of order
a divisor of
$(p\pm 1)/2$, then it follows that $WZ$ is abelian with cyclic Sylow
subgroups. Thus $WZ$ (and $W$) are cyclic of order dividing $p \pm 1$. Suppose
now that $H/Z$ is dihedral of order $2\cdot ({p \pm 1})$, where $H$ is a
$p'$-subgroup of $S$. Then $H/Z$ has a cyclic subgroup of index 2. Thus $H$
has a cyclic subgroup of index 2. Hence if $W$ is a subgroup of $H$, it has
a cyclic normal 2-complement $Q$, and a Sylow 2-subgroup $P$ such that
$|Q:\cent QP|\le 2$. Since $Q$ is cyclic, or generalised quaternion, it
follows that $W$ has a cyclic normal subgroup of index 2.
Suppose next that $WZ/Z=\fA_4$. Then $WZ$ has order 24, centre $Z$ of order 2,
and a unique involution.  Hence $WZ$ is $\SL_2(3)$. The proper subgroups
of $\SL_2(3)$ are cyclic or quaternion of order 8.
Suppose now that $WZ/Z=\fS_4$. Then $WZ$ has order 48, centre $Z$ of order 2
and a unique involution. Thus $WZ={\tt SmallGroup}(48,28)$. The proper
subgroups of this group are cyclic, have a cyclic normal subgroup
of index 2, or are isomorphic to $\SL_2(3)$. Finally, if $WZ/Z=\fA_5$, then
$WZ=\SL_2(5)$. The proper subgroups of $\SL_2(5)$ are already on our list.
\smallskip

(d)~~Write $V=C_p \times C_p$. Note that $\Irr(V)$ is then the dual of the
natural module for $U$, and $\SL_2(p)\le U$ acts transitively on
$\Irr(V)\setminus\{1\}$. Let $\theta\in\Irr(V)$. If $\theta=1$ then it
extends to $G$ and the constituents of $\theta^G$ are the inflations of
characters of $U$.
The ones of degree divisible by $p$ are thus the extensions of the Steinberg
character of $\SL_2(p)$ to $U$, all of degree~$p$. Now assume that
$\theta\ne1$. Then up to conjugation the inertia group $T$ of $\theta$ in $U$
contains all elements $\begin{pmatrix} 1& *\\ 0& *\end{pmatrix}$ in $U$, hence
$T$ is metacyclic with abelian Sylow subgroups and normal Sylow $p$-subgroup.
So $\theta$ extends to $V:T$, and the characters above $\theta$ have degrees
prime to~$p$. Since $|U:T|$ is prime to $p$, all characters in $\Irr(G|\theta)$
are of degree prime to $p$.

(e)~~Consider the canonical map $F:\Aut(L) \rightarrow \Aut(L/L')$.
Then $\Inn(L)$ lies in the kernel of $F$. Assume conversely that $f$ lies in
$\ker(F)$. If $L$ is generated by $x$ and $y$, we have at most $p^2$
possibilities for $f(x)$ in $xL'$ and $f(y)$ in $yL'$. On the other hand,
$|\Inn(L)|=p^2$. Hence, $\ker(F)=\Inn(L)$ and $F$ induces an embedding
$\Out(L)$ in $\GL_2(p)$. Winter \cite{Wi72} shows that $|\Out(L)|=p-1$.
The rest follows from the main theorem in \cite{Wi72}.
\end{proof}

\begin{lem}   \label{aux2}
 Let $S=\SL_2(p)$, where $p$ is an odd prime, and let $\alpha,\beta\in\Irr(S)$
 of degrees $\frac{p-1}{2}$ and $\frac{p+1}{2}$ such that
 $|\alpha(x) + \beta(x)|^2=|\cent V x|$ for all $x \in S$, where
 $V=C_p \times C_p$ is the natural module for $S$. Let $\Psi=\alpha + \beta$.
 \begin{enumerate}[\rm(a)]
  \item We have that $\Psi(x)=\pm 1$ for $p$-regular non-trivial $x \in S$.
   Furthermore, the values of $\Psi$ on $p$-regular elements do not depend of
   the choices of $\alpha$ and $\beta$.
  \item Let $\gamma\in\Irr(S)$ of degree $p$. Let $\Delta=\gamma\Psi$. Assume
   that $p\ge 7$. Then the irreducible $p$-Brauer constituents of
   $\Delta^0$ appear with multiplicity less than or equal to $(p-1)/2$, and
   $|\IBr(\Delta^0)|\le p$. This remains true when $p=5$, except that
   the multiplicity of the unique $\mu \in\IBr(S)$ of degree~3 is 3.
  \item Suppose that $1<W$ is a subgroup of $S$ of order not divisible by $p$.
   Suppose that $\psi$ is a character of $W$ such that $\psi(1)=p$ and
   $\psi(x)=\pm 1$ for $1 \ne x \in W$. Let $\theta,\delta\in\Irr(W)$. Then
   $$[\psi, \theta \delta] \le \frac{p}{2}$$
   unless $W=Q_8$ and $\theta=\delta$ has degree 2, or $|W|=2$. In these cases,
   $$[\psi, \theta \delta] \le \frac{p+1}{2} \, .$$
 \end{enumerate}
\end{lem}

\begin{proof}
Part~(a) follows immediately from the well-known character table of
$\SL_2(p)$.

(b) From the ordinary character table it can be worked out easily that
$\gamma\Psi$ is multiplicity free (and does not involve the trivial character
nor the characters of degree $\frac{p-1}{2}$). The Brauer trees of the
$p$-blocks of $S$ of positive defect have an exceptional node of
multiplicity~2, so any Brauer character occurs in at most three ordinary
characters. This shows the first claim for $p\ge7$. For $p=5$, direct
calculation suffices. The second assertion is immediate, as $l(S)=p$.

(c)~~If $W$ has order 2, our assertion easily follows because $\psi(1)=p$,
and $\psi(w)=\pm 1$ if $1 \ne w \in W$. Suppose that $W$ is a $p'$-subgroup
of $S$ with order $|W|>2$. We have checked (c) with GAP for primes
$3\le p\le 23$. Hence, we may assume that $p>23$, if necessary.

Now $W$ is one of the groups in Lemma~\ref{aux}(c). Let
$\theta,\delta\in\Irr(W)$. Then
$$[\psi, \theta\delta]=\frac{1}{|W|}
  \sum_{w \in W} \psi(w)\theta(w^{-1})\delta(w^{-1}) \le
  \frac{\theta(1)\delta(1) (|W|-1+p)}{|W| } \, ,$$
using that $|\theta(w)| \le \theta(1)$ for $\theta \in \Irr(W)$.

Suppose first that $\theta$ and $\delta$ are linear. Then we see that
$[\psi, \theta\delta] \le p/2$ if $p>3$. Thus we may assume that $W$ is
non-abelian.

Suppose now that $W$ has a normal abelian subgroup of index 2. If
$\theta,\delta\in\Irr(W)$, then $\theta(1)\delta(1)\le 4$. We see, assuming
that $p\ge 23$, that $[\psi_W, \theta\delta] \le p/2$ if $|W|\ge 12$.
There is only one non-abelian group of order less than 12 with a unique
involution, which is $Q_8$. If $\theta(1)\delta(1)\le 2$, then
$[\psi, \theta\delta]\le p/2$.
So we assume that $\theta=\delta$ has degree 2. Let $Z=\bZ(W)$.
Then using that $\theta$ is zero off $Z$, we have that
$$[\psi, \theta\delta]=\frac{1}{8} (4p \pm 4)=\frac{p \pm 1}{2} \, .$$

Suppose now that $W=\SL_2(3)$. The largest character degree of $W$ is~3, and
there is a unique character $\theta$ with that degree. Assuming that $p\ge 23$,
we have that $[\psi, \theta\delta] \le p/2$, if $\theta(1)\delta(1) \le 6$.
Assume now that $\theta=\delta$ has degree~3. This character has $Z$ in its
kernel, and otherwise takes value~0 except on the unique conjugacy class of
elements of order 4. On these six elements, $\theta$ has value $-1$. Hence
$$[\psi, \theta\delta] \le \frac{1}{24}(9p +9 + 6) \le p/2 \, $$
(if $p\ge 5$, which we are assuming).

Suppose now that $W={\tt SmallGroup}(48,28)$. This group has a unique
character $\theta \in \Irr(W)$ of degree 4. By using the values of this
character, we have that
$$[\psi, \theta^2] \le \frac{1}{48}(16p+32) \le p/2$$
for $p \ge 5$. If $\theta(1)\delta(1)=12$, then
$\theta\delta$ is zero except on $\bZ(W)$, and the inequality is clear.
If $\theta(1)\delta(1)=9$, we again use the character values to check the
inequality. Finally, if $\theta(1)\delta(1) \le 6$, then we do not need to use
the character values. The case $W=\SL_2(5)$ is done similarly.
\end{proof}

The character $\Psi$ in Lemma~~\ref{aux2} is relevant in the character theory
of fully ramified sections, as we shall see.

\begin{lem}   \label{aux3}
 Suppose that $L \nor G$ is an extra-special group of order $p^3$ and exponent
 $p$, where $p\ge 5$ is odd. Let $Z=\bZ(L) \sbs \bZ(G)$, and assume that
 $G/L \cong \SL_2(p)$ and that $\cent GL=Z$. Let $\alpha, \beta \in \Irr(G/L)$
 of degrees ${\frac{p-1}{2}}$ and $\frac{p+1}{2}$ such that
 $|\alpha(x) + \beta(x)|^2=|\cent V x|$ for all $x \in S$, where $V$ is the
 natural module for $S$. Let $1\ne\lambda\in\Irr(Z)$, and write
 $\lambda^L=p\eta$, where $\eta \in \Irr(L)$. Then $\eta$ has a unique
 extension $\hat\eta \in \Irr(G)$ and $\hat\eta^0=\alpha^0 + \beta^0$.
\end{lem}

\begin{proof}
Let $\Psi=\alpha + \beta$.
If $K/L$ is the centre of $G/L$, and $Q \in \Syl_2(K)$, then $N/Z$ is the
unique (up to $G$-conjugacy) complement of $L/Z$ in $G/Z$, where $N=\bN_G(Q)$.
This follows by the Frattini argument and the fact that $Q$ acts on $L/Z$ with
no non-trivial fixed points. Now, $N$ is a central extension of $\SL_2(p)$
so $N=Z \times N'$, where $N' \cong \SL_2(p)$. Since the Schur multiplier
of $G/L$ is trivial (by Lemma~\ref{aux}(b)), it follows that $\eta$ extends
to $G$ by \cite[Thm.~11.7]{Isa}. Since $G/L$ is perfect, there is a unique
extension $\hat\eta \in \Irr(G)$ by Gallagher's theorem.
Now, by \cite[Thm.~9.1]{Isa73}, there is a character $\psi$ of degree $p$
of $G/L$ such that
$$\hat\eta_N=\psi \nu$$
for some linear character $\nu$ of $N$ over $\lambda$. Since $N'$ is perfect,
we have that $\nu=\lambda \times 1_{N'}$. By the proof of
\cite[Thm.~4.8]{Isa73}, we have that $\psi=\Psi$, and therefore
$\hat\eta^0=\Psi^0$.
(Notice that there are two choices of $\Psi$, but $\Psi^0$ is uniquely
determined by Lemma~\ref{aux2}(a).)
\end{proof}

\begin{lem}   \label{multi}
 Suppose that $H\le G$ are finite groups. Let $\chi\in\Irr(G)$ and
 $\theta\in\Irr(H)$. Then we have $[\chi_H,\theta]^2\le|G:H|$.
\end{lem}

\begin{proof}
Write $\chi_H=e\theta + \Delta$, where $\Delta$ is a character of $H$ or zero,
and $[\Delta,\theta]=0$. Then $\chi(1) \ge e\theta(1)$. By Frobenius
reciprocity, we have that $\theta^G=e\chi + \Xi$, where $\Xi$ is a character
of $G$ or zero. Hence $|G:H|\theta(1) \ge e\chi(1) \ge e^2\theta(1)$, and the
proof is complete.
\end{proof}

\begin{thm}   \label{thm:p-solv}
 Suppose that $G$ is a finite group and let $\chi \in \Irr(G)$ be such that
 $\chi(1)_p=|G|_p/p^2$. Let  $N=\oh {p'}G$ and $L=\oh pG$. Suppose that
 $N \sbs \bZ(G)$, and that $\cent GL \sbs LN$. Then $|\IBr(\chi^0)| \le p^2-1$
 and $d_{\chi \vhi} \le p$ for all $\vhi \in \IBr(G)$.
\end{thm}

\begin{proof}
If $\chi^0 \in \IBr(G)$, then $d_{\chi \vhi} \le 1$ and $|\IBr(\chi^0)|=1$.
We may clearly assume that $\chi(1) >p$. Let $\theta \in \Irr(N)$ be the
irreducible constituent of the restriction $\chi_N$.

By hypothesis, we have that $\cent GL=\bZ(L) \times N$. In particular,
$L>1$ (because otherwise $G$ is a $p'$-group and $\chi^0 \in \IBr(G)$). Also,
we have that $G/\cent GL \cong U \sbs \Aut(L)$. Since $N$ is central,
then $\oh p{G/N}=LN/N$, and thus $\oh p U \cong L/\bZ(L)$.

Let $\eta \in \Irr(L)$ be under $\chi$. Then $\chi(1)_p/\eta(1)$ divides
$|G|_p/|L|$ by \cite[Cor.~11.29]{Isa}. Using the hypothesis, we have that
$|L|/\eta(1) \le p^2$. Since $\eta(1)<|L|^{1/2}$ because $L$ is a
non-trivial $p$-group, we deduce that $|L|\le p^3$. Also, $\eta(1) \le p$.

If $G/LN$ is a $p$-group, then $G=N\times L$.
Then $\chi=\theta \times \eta$. Thus $\chi^0=\eta(1)\theta$.
Hence $d_{\chi\vhi}=\eta(1)\le p$, and $|\IBr(\chi^0)|=1$, so
the theorem is true in this case too.

Suppose that $\theta$ extends to some $\gamma \in \Irr(G)$. Notice that
$\gamma^0 \in \IBr(G)$ also extends $\theta$. By Gallagher's Corollary 6.17
of \cite{Isa}, we know that $\chi=\beta \gamma$ for some $\beta\in\Irr(G/N)$.
If $\beta^0=d_1\tau_1 + \ldots + d_s\tau_s$, where
$\tau_i \in \IBr(G/N)$ are distinct, then we have that
$$\chi^0=d_1\tau_1 \gamma^0  + \ldots + d_s\tau_s \gamma^0$$
and that the $\tau_i\gamma^0$ are also distinct and irreducible (using that
$\gamma^0$ is linear). In particular, we see that if $\theta$ extends to $G$,
then the theorem holds for $G$ if it holds for $G/N$.
\medskip

(a)~~Assume first that $\eta(1)=1$. We have then that $|L|\le p^2$, by the
third paragraph of this proof. In particular, $L$ is abelian. In this case,
$G/LN \cong U \sbs \Aut(L)$ and $\oh pU=1$.
\smallskip

(a1)~~If $L$ is a Sylow $p$-subgroup of $G$, then by \cite[Thm.~10.20]{N},
we have that $\Irr(G|\theta)$ is a block of $G$, which has defect group $L$.
By the $k(GV)$-theorem, we have that $|\Irr(G|\theta)|\le p^2$. Hence
$|\IBr(G|\theta)|\le p^2-1$ by \cite[Thm.~3.18]{N}. Thus
$|\IBr(\chi^0)|\le p^2-1$. Also, notice that the decomposition numbers are
$$d=[\chi_H, \mu]$$
for $\mu \in \Irr(H)$, where $H$ is a $p$-complement of $G$.
By Lemma~\ref{multi}, $d^2\le|G:H|=p^2,$ and the theorem is true in this case.
\smallskip

(a2)~~We assume that $L$ is not a Sylow $p$-subgroup of $G$. Suppose first that
$L$ is cyclic. Notice that $L$ cannot have order $p$, since $p^2$ divides the
order of $G$ (and $|\Aut(C_p)|$ is not divisible by $p$). Suppose that
$L=C_{p^2}$. Then $G/LN$ is cyclic. Since $\oh p{G/LN}=1$, we conclude that
$G/LN$ is cyclic of order dividing $p-1$. Thus $L$ is a normal Sylow
$p$-subgroup of $G$, a contradiction.

Assume now that $L=C_p \times C_p$. In this case $G/LN\cong U \sbs \GL_2(p)$.
Thus $|U|_p= p$. Recall that $\oh pU=1$. By Lemma~\ref{aux}(a), we conclude
that $\SL_2(p) \sbs U \sbs \GL_2(p)$. Therefore $U$ has trivial Schur
multiplier by Lemma~\ref{aux}(b). Now, consider $\hat \theta=1_L\times\theta$.
By \cite[Thm.~11.7]{Isa}, we have that $\hat\theta$ extends to $G$. In
particular, so does $\theta$. Again, by the fifth paragraph of this proof,
we may assume in this case that $N=1$. If $p=2$, then $G=\fS_4$ and
$\chi(1)=2$, so the theorem is true in  this case.
If $p$ is odd, let $Z/L \sbs \bZ(G/L)$ of order~2, and let $Q \in \Syl_2(Z)$.
Since $Q$ acts on $L$ as the minus identity matrix, we have that $\cent LQ=1$.
Therefore $G$ is the semidirect product of $L$ with $\bN_G(Q)$. In this case,
by Lemma \ref{aux}(d), we have that $\chi(1)=p$, and we are also done in this
case.
\medskip

(b)~~Assume now that $\eta(1)=p$.
Hence $L$ is extraspecial of order $p^3$ and exponent $p$ or $p^2$.
Write $Z=\bZ(L)$. We have that $G/ZN\cong U \sbs \Aut(L)$ and
$\oh p U=C_p \times C_p$.
Also, write $\eta_Z=\eta(1) \lambda$, where $1\ne\lambda\in\Irr(Z)$.
\smallskip

(b1)~~Suppose first that $p=2$. Then $L=D_8$ or $Q_8$. If $L=D_8$, then
$\Aut(L)$ is a 2-group, and then $G=L \times N$, and the theorem is true in
this case. If $L=Q_8$, then $\Aut(Q_8)$ is $\fS_4$. Then $G/(L \times N)$ is
a subgroup of $\fS_3$. Then $\theta \times 1_L$ (and therefore $\theta$)
extends to $G$. By the fifth paragraph of this proof, we may therefore assume
that $N=1$. Thus $|G|\le 48$. In this case, $G=\SL_2(3), \GL_2(3)$ or the
{\sl fake} $\GL_2(3)$ ({\tt SmallGroup}(48,28)).
If $G=\SL_2(3)$, then $\chi(1)=2$, and we are done. In the remaining cases,
$\chi(1)=4$, $|\IBr(\chi^0)|=2$ and $d_{\chi \vhi}=1$ or 2.
Hence, we may assume that $p$ is odd.
\smallskip

(b2)~~Suppose first that $L$ is extra-special of exponent $p^2$, $p$ odd.
By \cite{Wi72}, we have that $\Aut(L)=X:C_{p-1}$, were $|X|=p^3$ has as a
normal subgroup $\Inn(L)$, and $C_{p-1}$ acts Frobenius on $Z$. Thus
$G/(L \times N)$ has cyclic Sylow subgroups.
It follows that $1_L \times \theta$ (and therefore $\theta$) extends to $G$.
Recall that $\oh p{G/Z}=L/Z$. Since $G/L$ is a subgroup of $C_p:C_{p-1}$ and
this group has a normal Sylow $p$-subgroup, it follows that $G/L$ is cyclic
of order dividing $p-1$. Now, $G$ is the semidirect product of $L$ with a
cyclic group $C$ of order $h$ dividing $p-1$ that acts Frobenius on $Z$.
Since $C$ acts Frobenius on $Z$, it follows that the stabiliser
$I_G(\lambda)=L$. Since $\lambda^L=p\eta$, it follows that $I_G(\eta)=L$.
Hence $\chi=\eta^G$. Then $\chi_C=p \rho$, where $\rho$ is the regular
character of $C$. Thus $d_{\chi \vhi}=p$ for every $\vhi \in \IBr(G)$
and $|\IBr(\chi^0)|=h \le p-1$. The theorem follows in this case too.

\smallskip
(b3)~~So finally assume that $L$ is extra-special of exponent $p$, $p$ odd.
We have that $\Aut(L)/\Inn(L)=\GL_2(p)$ by Lemma~\ref{aux}(e).
Thus $G/LN$ is isomorphic to a subgroup $W$ of $\GL_2(p)$ with $\oh p W=1$.
Write $C=\cent GZ$. Notice that $C/LN$ maps into $\SL_2(p)$ by \cite{Wi72},
and therefore this group has at most one involution. Also, $G/C$ is a cyclic
$p'$-group (because $\GL_2(p)/\SL_2(p)$ is) that acts Frobenius on $Z$.
Thus, our group $W$ has a normal subgroup that has at most one involution and
with cyclic $p'$-quotient. Also notice that $C$ is the stabiliser of $\lambda$
(and of $\eta$) in $G$. Furthermore, we claim that the stabiliser of any
$\gamma \in \Irr(C/LN)$ in $G$ has index at most 2 in $G$. This is because
$\GL_2(p)$ has a centre of order $p-1$ that intersects with $\SL_2(p)$ in its
unique subgroup of order 2. Write $\chi=\mu^G$, where $\mu \in \Irr(C)$.
\smallskip

(b.3.1)~~If $L$ is a Sylow $p$-subgroup of $G$, we claim that
$|\IBr(\chi^0)|\le p^2-1$. Let $H$ be a $p$-complement of $G$. Then $H/N$ is
isomorphic to a $p'$-subgroup of $\GL_2(p)$. By the $k(GV)$-theorem
applied to $\Gamma=(C_p \times C_p):H/N$, we have that $k(\Gamma)\le p^2$.
Since $\Gamma$ has a unique $p$-block, it follows that $k(H/N)\le p^2-1$.
Now, $|\IBr(\chi^0)| \le |\Irr(H|\theta)| \le k(H/N)$, and the claim is proved.
(The last inequality follows from Problem 11.10 of \cite{Isa}.)
\smallskip

(b.3.2)~~If $L$ is a Sylow $p$-subgroup of $G$, then $d_{\chi \vhi} \le p$:

Let $H$ be a $p$-complement of $G$. Write $C\cap H=Q$.
Also, $U=HZ \cap C$ is the unique complement of $L/Z$ in $C/Z$ up to
$C$-conjugacy. By Lemma \ref{aux}(b) (and \cite[Thm.~11.7]{Isa}), we have
that $1_L \times \theta$ has a (linear) extension $\tilde\theta \in \Irr(C)$.
Therefore $\mu=\beta \tilde\theta$, where $\beta \in \Irr(C/N)$. Notice that
$\beta$ lies over $\eta$. By \cite[Thm.~9.1]{Isa73}, we have that
$$\beta_U=\psi \beta_0 $$
for some $\beta_0 \in \Irr(U/N)$, where $\psi$ is a character of $C/N$
of degree $p$. Since $\cent{L/Z}x$ is trivial for non-trivial
$p$-regular $xN$, it follows that
$\psi(xN)=\pm 1$ for $p$-regular $Nx \ne N$. (The values of $\psi$ are
given in page 619 of \cite{Isa73}, and it can be checked that $\psi$
is the restriction of the character $\Psi$ of $\SL_2(p)$
given in Lemma~\ref{aux2}(a) under suitable identification.)
Notice that $(\beta_0)_Q \in \Irr(Q/N)$ because $Q$ is central in $U=ZQ$.
Now, let $\nu \in \Irr(Q|\theta)$. By Gallagher, we have that
$\nu=\tilde\theta_Q \tau$ for some $\tau \in \Irr(Q/N)$. Then
$$[\mu_Q, \nu]=[\beta_Q \tilde\theta_Q, \tilde\theta_Q \tau]
  =[\psi_Q (\beta_0)_Q, \tau] \,.$$
Now, by Lemma~\ref{aux2}(c) applied in $Q/N$ this number is less than $p/2$,
except if $Q/N=Q_8$ and $\tau$ is the unique character of degree~2.
In this case,  this number is less than $\frac{p+1}{2}$, and  $\tau$ extends
to $H$ because $H/Q$ is cyclic. Hence, if $\rho \in \Irr(H|\theta)$, then
$$[\chi_H, \rho]=[\mu_Q, \rho_Q] \, .$$
If $\rho_Q$ is irreducible, then we are done.
Let $\tau \in \Irr(Q|\theta)$ be under $\rho$.
We know that the stabiliser $I$ of $\tau$ in $H$ has index at most 2.
If $I=H$, then $\rho_Q$ is irreducible (because $H/Q$ is cyclic).
We conclude that $|H:I|=2$. Using that $I/Q$ is cyclic, we have that
$\rho_Q=\tau + \tau^x$, where $x \in H\setminus I$. In this case,
$$[\chi_H, \rho]=[\mu_Q, \tau + \tau^x] \le p/2  + p/2=p \,.$$
\smallskip

(b.3.3)~~We may assume that $p>3$: Else we have that $G/LN$ is a subgroup
of $\GL_2(3)$. All subgroups of $\GL_2(3)$ have trivial Schur multiplier
except for $C_2 \times C_2$, $D_8$ and $D_{12}$. By the requirements in (b3),
only $C_2 \times C_2$ can occur. But then, $L$ is a normal Sylow $p$-subgroup,
and again the theorem holds in this case.
\smallskip

(b.3.4)~~Suppose finally that $p$ divides $|W|$, and that $p>3$. Then
$\SL_2(p) \sbs W \sbs \GL_2(p)$ by Lemma~\ref{aux}(a). Now, by considering the
character $1_L \times \theta$ and using Lemma~\ref{aux}(b), we may again
assume that $\theta$ extends to $G$, and therefore that $N=1$ in this case.
Now, $C=\cent GZ$ is such that $\SL_2(p) \cong C/L \nor G/L$. By
Lemma~\ref{aux3}, we have that $\eta$ has a unique extension
$\hat\eta\in\Irr(C)$. Hence, $\mu$, the Clifford correspondent of $\chi$ over
$\eta$ is such that $\mu=\gamma \hat\eta$ for a unique $\gamma \in \Irr(C/L)$
of degree~$p$. Also, by Lemma \ref{aux3}, we know that $\hat\eta^0 = \Psi^0$.
Therefore, by Lemma~\ref{aux2}(b), we have that
$$\mu^0=d_1 \vhi_1 + \ldots + d_k \vhi_k$$
for some distinct $\vhi_i \in \IBr(C/L)$, with $k \le p$, and $d_i < p/2$,
except in the case where $p=5$. In this latter case, we still have
that $k\le p$ and that $d_i \le p/2$, except for the unique irreducible
$p$-Brauer character of degree 3, call it $\vhi_1$, which  is such that
$d_1=3$.  Now, since $G/C$ is cyclic of $p'$-order and $|\bZ(\GL_2(p))|=p-1$,
it follows that the stabiliser $T_i$ of
$\vhi_i$ in $G$ has index at most 2. Also $\vhi_i$ extends to $T_i$
(\cite[Cor.~8.12]{N}) and the irreducible constituents of $(\vhi_i)^{T_i}$
all appear with multiplicity 1, by \cite[Thm.~8.7]{N} and Gallagher's theorem
for Brauer characters \cite[Thm.~8.20]{N}. Also, they all induce irreducibly
to $G$ by the Clifford correspondence for Brauer characters.
Hence, the number of Brauer irreducible constituents of $\chi^0=(\mu^0)^G$
is less than or equal to $k\cdot (p-1) \le p^2-p$, and the decomposition
numbers are at most $2d_i \le p$, except if $p=5$ and $i=1$. 
In this case, $\vhi_1$ is $G$-invariant, because it is the
unique irreducible $p$-Brauer character of $\SL_2(5)$ of degree~3. Hence,
for $j>1$ the Brauer character $\vhi_j^G$ does not contain any irreducible
constituent of $\vhi_1^G$. So the irreducible constituents
of $\vhi_1^G$ appear in $\chi^0$ with multiplicity $3<5=p$.
\end{proof}

Now, we can finally prove the $p$-solvable case of Conjecture A.

\begin{cor}   \label{cor:p-solv}
 Suppose that $G$ is a finite $p$-solvable group and let $\chi \in \Irr(G)$ be
 such that $\chi(1)_p=|G|_p/p^2$. Then $|\IBr(\chi^0)| \le p^2-1$ and
 $d_{\chi \vhi}\le p$ for all $\vhi \in \IBr(G)$.
\end{cor}

\begin{proof}
Write $\oh{p'}G=N$. We argue by induction on  $|G: N|$.
Let $\theta \in \Irr(N)$ be under $\chi$.

Let $T$ be the inertia group of $\theta$ in $G$ and let
$\psi \in \Irr(T|\theta)$ be the Clifford correspondent of $\chi$ over
$\theta$. Since $\psi^G=\chi$, then we know that $d_\psi=2$.
Now we apply the Fong--Reynolds Theorem 9.14 of \cite{N} to conclude that
we may assume that $\theta$ is $G$-invariant.

By using ordinary/modular character triples (see Problem 8.13 of \cite{N}), we
may replace $(G,N,\theta)$ by some other  triple $(\Gamma,M, \lambda)$,
where $M=\oh{p'}\Gamma\sbs \bZ(\Gamma)$. Hence, by working now in $\Gamma$,
it is no loss to assume  that $N \sbs \bZ(G)$.
Now, if $L=\oh pG$, then we have that $\cent GL \subseteq LN$, and we may apply
Theorem~\ref{thm:p-solv} to conclude.
\end{proof}

\section{Proof of Theorem C}

A well-known theorem by Taussky asserts that a non-abelian $2$-group $P$ has
maximal (nilpotency) class if and only if $|P/P'|=4$ (see
\cite[Satz~III.11.9]{Huppert}). In this case $P$ is a dihedral group, a
semidihedral group or a quaternion group. Of course, $|P/P'|$ is the number of
linear characters of $P$. Our next result indicates that Taussky's theorem
holds for blocks, assuming the Alperin--McKay conjecture. For this we need to:

\begin{lem}   \label{lem}
 Let $B$ be a $2$-block of $G$ with defect group $D$. Suppose that $B$
 satisfies the Alperin--McKay conjecture. Then $k_0(B)=4$ if and only if $D$ is
 dihedral (including Klein four), semidihedral, quaternion or cyclic of
 order~$4$.
\end{lem}

\begin{proof}
If $D$ is one of the listed groups, then $k_0(B)=4$ by work of Brauer 
and Olsson (see \cite[Thm.~8.1]{Sambale}). 
Now suppose conversely that $k_0(B)=4$.
Since $B$ satisfies the Alperin--McKay conjecture, we may assume that
$D\unlhd G$. By Reynolds~\cite{Reynolds}, we may also assume that $D$ is a
Sylow $2$-subgroup of $G$. By \cite[Thm.~6]{KuelshammerRemark}, $B$ dominates
a block $\overline{B}$ of $G/D'$ with defect group $\overline{D}:=D/D'$ and
$k(\overline{B})=k_0(B)=4$. Using Taussky's theorem, it suffices to show
$|\overline{D}|=4$.

By way of contradiction, suppose that $2^d:=|\overline{D}|>4$. Let
$(1,\overline{B})=(x_1,b_1),\ldots,(x_r,b_r)$ be a set of representatives for
the conjugacy classes of $\overline{B}$-subsections. Then
\[\sum_{i=1}^r{l(b_i)}=k(\overline{B})=4.\]
By \cite[Thm.~A]{Kk4}, we have $l(\overline{B})\ge 2$ and $r\le 3$. Let $I$
be the inertial quotient of $\overline{B}$. Then $I$ has odd order and so is
solvable by Feit--Thompson. The case $r=3$ is impossible, since $2^d$ is the
sum of $I$-orbit lengths. Hence, $r=2$ and $\overline{D}$ is elementary
abelian.

Since $\overline{D}\rtimes I$ is a solvable $2$-transitive group, Huppert
\cite{Hupperttlg} implies that $I$ lies in the semilinear group
$\operatorname{\Gamma L}_1(2^d)\cong C_{2^d-1}\rtimes C_d$. Let $N\unlhd I$
with $N\le C_{2^d-1}$ and $I/N\le C_d$. Then $N$ acts semiregularly on
$\overline{D}\setminus\{1\}$. Hence, $C_I(x_2)$ is cyclic (as a subgroup of
$I/N$). On the other hand, $C_I(x_2)$ is the inertial quotient of $b_2$. It
follows that $l(b_2)=|C_I(x)|$ is odd and therefore $l(b_2)=1$. Consequently,
$I$ acts regularly on $\overline{D}\setminus\{1\}$. This implies that all Sylow
subgroups of $I$ are cyclic. Hence, the K\"ulshammer--Puig class of
$\overline{B}$ is trivial and we obtain the contradiction
$l(\overline{B})=k(I)>3$.
\end{proof}

Our next result includes the first part of Theorem~C.

\begin{thm}   \label{thm:p2}
 Let $B$ be a $2$-block of $G$ satisfying the Alperin--McKay conjecture. If
 there exists a character $\chi\in\Irr(B)$ such that $\chi(1)_2=|G|_2/4$, then
 $\lvert\IBr(\chi^0)\rvert\le l(B)\le 3$ and $d_{\chi\phi}\le 2$ for all
 $\phi\in\IBr(B)$.
\end{thm}

\begin{proof}
By a result of Landrock (see \cite[Prop.~1.31]{Sambale}), $k_0(B)=4$. Hence,
Lemma~\ref{lem} applies. If $B$ has defect $2$, then $B$ is Morita equivalent
to the principal block of a defect group $D$ of $B$, of $\fA_4$ or $\fA_5$.
The claim follows easily in this case. Thus, we may assume that $B$ has defect
at least $3$. Then by work of Brauer and Olsson (see \cite[Thm.~8.1]{Sambale}),
$l(B)\le 3$. The claim about the decomposition numbers follows from the tables
at the end of \cite{Erdmann}.
\end{proof}

\begin{rem}
For every defect $d\ge 2$ there are $2$-blocks with defect $d$ containing an
irreducible character $\chi$ such that $\chi(1)_2=|G|_2/4$. This is clear for
$d=2$ and for $d\ge 3$ one can take the principal block of $\SL_2(q)$ where $q$
is a suitable odd prime power. These blocks have quaternion defect groups.
Similarly, the principal $2$-block of $\GL_2(q)$ where $q\equiv 3\pmod{4}$
gives an example with semidihedral defect group. On the other hand, Brauer
showed that there are no examples with dihedral defect group of
order at least $16$ (see \cite[Theorem~8.1]{Sambale}).
\end{rem}

In order to say something about odd primes, we need to invoke a stronger
conjecture known as Robinson's \emph{ordinary weight conjecture} (see
\cite[Conj.~2.7]{Sambale}). 
Robinson gave the following consequence of his conjecture
which is relevant to our work.

\begin{lem}[{\cite[Lemma~4.7]{Robinsonmetac}}]   \label{robinson}
 Let $B$ be a $p$-block of $G$ with defect group $D$ satisfying the ordinary
 weight conjecture.
 Assume that there exists $\chi\in\Irr(B)$ such that $\chi(1)_p=|G|_p/p^2$.
 Then $|D|=p^2$ or $D$ has maximal class. Let $\cF$ be the fusion
 system of $B$. If $|D|\ge p^4$, then $D$ contains an $\cF$-radical,
 $\cF$-centric subgroup $Q$ of order $p^3$ such that
 $\SL_2(p)\le\Out_{\cF}(Q)\le\GL_2(p)$.
 In particular, $Qd(p)$ is involved in $G$. If $p=2$, then $Q\cong Q_8$ and if
 $p>2$, then $Q$ is the extraspecial group $p^{1+2}_+$ with exponent $p$.

 Conversely, if $B$ is any block (satisfying the ordinary weight conjecture)
 with an $\cF$-radical, $\cF$-centric subgroup $Q$ as above,
 then $\Irr(B)$ contains a character $\chi$ with $\chi(1)_p=|G|_p/p^2$.
\end{lem}

Observe that a $p$-group $P$ of order $|P|\ge p^3$ has maximal class if and
only if there exists $x\in P$ with $|C_P(x)|=p^2$ (see
\cite[Satz~III.14.23]{Huppert}). We will verify that latter condition in a
special case in Proposition~\ref{prop:smallcent}.

Lemma~\ref{robinson} implies for instance that the group
$p^{1+2}_+\rtimes\SL_2(p)$ contains irreducible characters $\chi$ with
$\chi(1)_p=p^2$. Hence, for every prime $p$ there are $p$-blocks of defect $4$
with characters of defect $2$. As another consequence we conditionally extend
Landrock's result mentioned in the proof of Theorem~\ref{thm:p2} to odd primes.

\begin{prop}   \label{propk0}
 Let $B$ be a $p$-block of $G$ satisfying the ordinary weight conjecture.
 If there exists $\chi\in\Irr(B)$ such that $\chi(1)_p=|G|_p/p^2$, then
 $k_0(B)\le p^2$.
\end{prop}

\begin{proof}
By Lemma~\ref{robinson}, a defect group $D$ of $B$ has maximal class or
$|D|=p^2$. In any case $|D/D'|=p^2$. The ordinary weight conjecture implies the
Alperin--McKay conjecture (blockwise) and by \cite{KuelshammerRemark}, the
Alperin--McKay conjecture implies Olsson's conjecture $k_0(B)\le|D/D'|=p^2$.
\end{proof}

If we also assume the Eaton-Moret\'o conjecture~\cite{EatonMoreto} for $B$, it
follows that $k_1(B)>0$ in the situation of Proposition~\ref{propk0}. This is
because a $p$-group $P$ of maximal class has a (unique) normal subgroup $N$
such that $P/N$ is non-abelian of order $p^3$. Hence, $P$ has an irreducible
character of degree $p$.

\begin{prop}   \label{psolv}
 Let $B$ be a $p$-block of a $p$-solvable group $G$ with $\chi\in\Irr(B)$ such
 that $\chi(1)_p=|G|_p/p^2$. Then one of the following holds:
 \begin{enumerate}[\rm(1)]
  \item $B$ has defect $2$ or $3$.
  \item $p=2=l(B)$ and $B$ has defect group $Q_{16}$ or $SD_{16}$. Both cases
   occur.
  \item $p=3$ and $B$ has defect group
   \textnormal{\texttt{SmallGroup}}$(3^4,a)$ with $a\in\{7,8,9\}$. All three
   cases occur.
 \end{enumerate}
\end{prop}

\begin{proof}
By Haggarty~\cite{Haggarty}, $B$ has defect at most $4$. Moreover, if $B$ has
defect $4$, then $p\le3$. Let $D$ be a defect group of $B$.
Since the ordinary weight conjecture holds for $p$-solvable groups,
Lemma~\ref{robinson} implies that $D$ has maximal class. If $p=2$, then the
fusion system $\cF$ of $B$ contains an $\cF$-radical, $\cF$-centric subgroup
isomorphic to $Q_8$. Hence, $D\in\{Q_{16},SD_{16}\}$. On the other hand, every
fusion system of a block of a $p$-solvable group is constrained. This implies
that there is only one $\cF$-radical, $\cF$-centric subgroup in $D$. It follows
from \cite[Thm.~8.1]{Sambale} that $l(B)=2$. Examples are given by the two
double covers of $\fS_4$.

Now suppose that $p=3$. According to GAP~\cite{GAP}, there are four
possibilities for $D$: $\texttt{SmallGroup}(3^4,a)$ with $a\in\{7,8,9,10\}$.
In case $a=10$, $D$ has no extraspecial subgroup of order $27$ and exponent~3.
Hence, Lemma~\ref{robinson} excludes this case. Conversely, examples for
the remaining three cases are given by the (solvable) groups
$\texttt{SmallGroup}(3^4\cdot8,b)$ with $b\in\{531,532,533\}$.
\end{proof}

In the following we (conditionally) classify the possible defect groups in
case $p=3$. This relies ultimately on Blackburn's classification of the
$3$-groups of maximal class. Unfortunately, there is no such classification
for $p>3$.

\begin{prop}   \label{prop:p3}
 Let $B$ be a $3$-block of $G$ satisfying the ordinary weight conjecture.
 Suppose that there exists $\chi\in\Irr(B)$ such that $\chi(1)_3=|G|_3/9$.
 If $B$ has defect $d\ge 4$, then are at most three possible defect groups of
 order $3^d$ up to isomorphism. If $d$ is even, they all occur, and if $d$ is
 odd, only one of them occurs. In particular, we have examples for every
 defect $d\ge 2$.
\end{prop}

\begin{proof}
In case $d=4$ we can argue as in Proposition~\ref{psolv}. Thus, suppose that
$d\ge 5$. By Lemma~\ref{robinson}, a defect group $D$ of $B$ has maximal class
and contains a radical, centric subgroup $Q\cong 3^{1+2}_+$.
In particular, $B$ is not a controlled block. Since $d\ge 5$, it is known that
$D$ has $3$-rank $2$ (see \cite[Thm.~A.1]{tworank}). Hence, the possible
fusion systems $\cF$ of $B$ are described in \cite[Thm.~5.10]{tworank}. It
turns out that $D$ is one of the groups $B(3,d;0,\gamma,0)$ with
$\gamma\in\{0,1,2\}$. If $d$ is odd, then $\gamma=0$. In all these cases
examples are given such that $\Out_{\cF}(Q)\cong\SL_2(3)$. We can pick for
instance the principal 3-blocks of $3.\PGL_3(q)$ and $\tw2F_4(q)$ for a
suitable prime power $q$. An inspection of the character tables in
\cite{Steinberg,Ma2F4} shows the existence of $\chi$.
\end{proof}

Now we are in a position to cover the second part of Theorem~C (recall that 
the ordinary weight conjecture for all blocks of all finite groups implies 
Alperin's weight conjecture).

\begin{cor}
 Let $B$ be a $3$-block of $G$ with defect $d$ satisfying the ordinary weight
 conjecture and Alperin's weight conjecture.
 If there exists $\chi\in\Irr(B)$ such that $\chi(1)_3=|G|_3/9$, then
 \[l(B)\le\begin{cases}
   8&\text{if }4\ne d\equiv 0\pmod{2},\\
   9&\text{if }d\equiv 1\pmod{2},\\
   10&\text{if }d=4.
  \end{cases}\]
\end{cor}

\begin{proof}
Let $D$ be a defect group of $B$. We may assume that $|D|\ge 27$. Then $D$ has
maximal class. In case $|D|=27$ and $\exp(D)=9$, Watanabe has shown that 
$l(B)\le 2$ without invoking any conjecture (see
\cite[Thm.~1.33 and 8.8]{Sambale}). Now assume that
$D\cong 3^{1+2}_+$. Then the possible fusion systems $\cF$ of $B$ are given in
\cite{ExtraspecialExpp}. To compute $l(B)$ we use Alperin's weight conjecture
in the form \cite[Conj.~2.6]{Sambale}. Let $Q\le D$ be $\cF$-radical and
$\cF$-centric. For $Q=D$ we have $\Out_{\cF}(D)\le SD_{16}$. Hence, regardless
of the K\"ulshammer--Puig class, $D$ contributes at most $7$ to $l(B)$. For
$Q<D$ we have $\Out_{\cF}(Q)\in\{\SL_2(3),\GL_2(3)\}$. The groups $\SL_2(3)$
and $\GL_2(3)$ have trivial Schur multiplier and exactly one respectively two
irreducible characters of $3$-defect $0$. Hence, each $\cF$-conjugacy class
of such a subgroup $Q$ contributes at most $2$ to $l(B)$. There are at most two
such subgroups up to conjugation.
Now an examination of the tables in \cite{ExtraspecialExpp} yields the claim
for $d=3$. Note that $l(B)=9$ only occurs for the exceptional fusion systems on
$\tw2F_4(2)'$ and $\tw2F_4(2)$.

Now let $d\ge 5$.
As in Proposition~\ref{prop:p3} there are at most three possibilities for $D$
and the possible fusion systems $\cF$ are listed in \cite[Thm.~5.10]{tworank}.
We are only interested in those cases where there exists an $\cF$-radical,
$\cF$-centric, extraspecial subgroup of order $27$. We have
$\Out_{\cF}(D)\le C_2\times C_2$. Hence, $D$ contributes at most $4$ to $l(B)$.
Now assume that $Q<D$ is $\cF$-radical and $\cF$-centric (i.\,e., $\cF$-Alperin
in the notation of \cite{tworank}). Then as above,
$\Out_\cF(Q)\in\{\SL_2(3),\GL_2(3)\}$.
There are at most three such subgroups up to conjugation. The claim
$l(B)\le 9$ follows easily.
In case $l(B)=9$, $\cF$ is the fusion system of $^2F_4(q^2)$ for some $2$-power
$q^2$ or $\cF$ is exotic. In both cases $d$ is odd. The principal block of
$\tw2F_4(q^2)$ shows that $l(B)=9$ really occurs (see \cite{Ma2F4}).

It remains to deal with the case $d=4$. By the results of \cite{tworank}, we
may assume that $D\cong C_3\wr C_3$. The fusion systems on this group seem to
be unknown. Therefore, we have to analyse the structure of $D$ by hand.
Up to conjugation, $D$ has the following candidates of $\cF$-radical,
$\cF$-centric subgroups: $Q_1\cong C_3\times C_3$, $Q_2\cong 3^{1+2}_+$,
$Q_3\cong C_3\times C_3\times C_3$ and $Q_4=D$. We may assume that
$Q_1\le Q_2$. As before, $Q_2$ must be $\cF$-radical and
$\Out_{\cF}(Q_2)\in\{\SL_2(3),\GL_2(3)\}$. Hence, $Q_1$ is conjugate to $D'$
under $\Aut_{\cF}(Q_2)$. Since $C_D(D')=Q_3$, we conclude that $Q_1$ is not
$\cF$-centric.
Now let $\alpha\in\Aut_{\cF}(Q_2)$ the automorphism inverting the elements of
$Q_2/Z(Q_2)$. Then $\alpha$ acts trivially on $Z(Q_2)$. By the saturation
property of fusion systems, $\alpha$ extends to $D$. Since $Q_3$ is the only
abelian maximal subgroup of $D$, the extension of $\alpha$ restricts to $Q_3$.
Since $Z(Q_2)\le Q_3$, it follows from a GAP computation that
$\Out_{\cF}(Q_3)\in\{\fS_4,\fS_4\times C_2\}$. Using similar arguments we end
up with two configurations:
\begin{enumerate}
\item[(i)] $\Out_{\cF}(Q_2)\cong\SL_2(3)$, $\Out_{\cF}(D)\cong C_2$ and
  $\Out_{\cF}(Q_3)\cong \fS_4$.
\item[(ii)] $\Out_{\cF}(Q_2)\cong\GL_2(3)$, $\Out_{\cF}(D)\cong C_2\times C_2$
  and $\Out_{\cF}(Q_3)\cong \fS_4\times C_2$.
\end{enumerate}
In the first case we have $l(B)\le 5$ (occurs for the principal block of
$\PSL_4(4)$) and in the second case $l(B)\le 10$ (occurs for the principal
block of $\PSL_6(2)$).
\end{proof}

Concerning the primes $p>3$ we note that for example the principal $5$-block of
$\PSU_6(4)$ has defect $6$ and an irreducible character of defect $2$. However,
we do not know if for any prime $p\ge5$ and any $d\ge5$ there are
$p$-blocks of defect $d$ with irreducible characters of defect~$2$.

\section{Symmetric, alternating and sporadic groups}

In this section we discuss the validity of our conjectures for alternating,
symmetric and sporadic groups.

\subsection{}
The irreducible characters of the symmetric group $\fS_n$ are parametrised by
partitions $\la$ of $n$, and we shall write $\chi_\la\in\Irr(\fS_n)$ for the
character labelled by $\la$. Its degree is given by the well-known hook
formula.
We first address Conjecture~B.

\begin{prop}   \label{prop:lB-Sn}
 Conjecture~B holds for the alternating and symmetric groups at any prime.
\end{prop}

\begin{proof}
First consider $G=\fS_n$. Let $p$ be a prime. For $\chi=\chi_\la\in\Irr(\fS_n)$
let $w\ge0$ be its defect, i.e., such that $p^w\chi(1)_p=|\fS_n|_p$. It is
immediate from the hook formula that this can only happen if there are at most
$w$ ways to move a bead upwards on its respective
ruler in the $p$-abacus diagram of $\la$. But then clearly the $p$-core of
$\la$ can be reached by removing at most $w$ $p$-hooks, so $\la$ lies in a
$p$-block of weight at most~$w$. But it is well-known that any such block
$B$ has less than $p^w$ modular irreducible characters (see
e.g.~\cite[Prop.~5.2]{MR16}).
\par
We now consider the alternating groups where we first assume that $p$ is odd.
Clearly
$\chi\in\Irr(\fA_n)$ satisfies our hypothesis if and only if it lies below a
character $\chi_\la$ of $\fS_n$ which does. But then $\chi_\la$ lies in a
$p$-block $B$ of $\fS_n$ of weight at most~$w$, as shown before. If $B$ is
not self-associate, that is, if the parametrising $p$-core is not self-dual,
then $B$ and its conjugate $B'$ both lie over a block $B_0$ of $\fA_n$, namely
the one containing $\chi$, with the same invariants. So we are done
by the case of $\fS_n$. If $B$ is self-associate then an easy estimate shows
that again $l(B)<p^w$ (see the proof of \cite[Prop.~5.2]{MR16}).
\par
Finally consider $p=2$ for alternating groups. Here the number of modular
irreducibles in a 2-block $B$ of weight~$w$ is $\pi(w)$ if $w$ is odd,
respectively $\pi(w)+\pi(w/2)$ if $w$ is even, with $\pi(w)$ denoting the
number of partitions of $w$. Now $\pi(w)\le 2^{w-1}$, and moreover
$\pi(w)+\pi(w/2)\le 2^{w-1}$ for even $w\ge4$, so for $w\ge3$ we have
$l(B)\le 2^{w-1}$. On the other hand the difference $d(B)-\Ht(\chi)$ is
at most one smaller for $\chi\in\Irr(\fA_n)$ than for a character $\tilde\chi$
of $\fS_n$ lying above $\chi$. Thus for all $w\ge3$ we have $l(B)\le 2^{w-1}$
is at most $2^{d(B)-\Ht(\chi)}$ for all $\chi\in\Irr(B)$, as required. For
$w=2$ the defect groups of $B$ are abelian and the claim is easily verified.
\end{proof}

The next statement follows essentially from a result of Scopes:

\begin{thm}   \label{thm:An}
 Conjecture~A holds for the alternating and symmetric groups
 at any prime.
\end{thm}

\begin{proof}
We first consider symmetric groups. Let $\chi=\chi_\la\in\Irr(\fS_n)$ be such
that $p^2\chi(1)_p=|\fS_n|_p$ for some prime~$p$. By the hook formula this
happens if and only if exactly two beads can be moved in exactly one way on
their respective
ruler in the $p$-abacus diagram of $\la$, or if one bead can be moved in two
ways. In either case the $p$-core of $\la$ can be reached by removing two
$p$-hooks, so $\la$ lies in a $p$-block of weight~2.

If $B$ is a $p$-block of weight~2 of $\fS_n$, with $p$ odd, then all
decomposition numbers are either~0 or~1 by Scopes \cite[Thm.~I]{Sco95}, and
furthermore any row in the decomposition matrix has at most 5 non-zero
entries. This proves our claim for $\fS_n$ and odd primes.
For $p=2$ the defect groups of $B$ of weight~2 are dihedral of order~8, and our
claim also follows.
\par
We now consider the alternating groups. First assume that $p$ is odd. Then
$\chi\in\Irr(\fA_n)$ satisfies our hypothesis if and only if it lies below a
character $\chi_\la$ of $\fS_n$ which does. But then $\chi_\la$ lies in a
$p$-block of $B$ of $\fS_n$ of weight~2, as shown before. If $B$ is not
self-associate then
$B$ and its conjugate $B'$ both lie over a block $B_0$ of $\fA_n$ with the same
invariants, in particular, with the same decomposition numbers. So we are done
by the case of $\fS_n$. If $B$ is self-associate then it is easy to count
that $B$ contains $(p-1)/2$ self-associate characters and $(p+1)^2/2$ that are
not. That is, $(p+1)^2/2$ characters in $B$ restrict irreducibly, while
$(p-1)/2$ of them split. It is clear from the $\fS_n$-result that the
decomposition numbers are at most two, and $l(B)\le(p^2+6p-3)/4<p^2$, so the
conjecture holds.
\par
Finally the case $p=2$ for $\fA_n$ follows by Theorem~\ref{thm:p2}, as the
Alperin--McKay conjecture is known to hold for all blocks of $\fA_n$,
see \cite{MO90}.
\end{proof}

The case of faithful blocks for the double covering groups of $\fA_n$ and
$\fS_n$ seems considerably harder to investigate, at least in as far as
decomposition numbers are concerned, due to the missing analogue of the
Theorem of Scopes for this situation.

\begin{prop}   \label{prop:2.An}
 Conjecture~B holds for the 2-fold covering groups of alternating and
 symmetric groups at any odd prime, and Conjecture A holds for these
 groups at $p=2$.
\end{prop}

\begin{proof}
As the Alperin--McKay conjecture has been verified for all blocks of the
covering groups of alternating and symmetric groups \cite{MO90}, our claim
for the prime $p=2$ follows from Theorem~\ref{thm:p2}.
\par
So now assume that $p$ is an odd prime, and first consider $G=2.\fS_n$, a
2-fold covering group of $\fS_n$, with $n\ge5$. By the hook formula for
spin characters \cite[(7.2)]{Ol93}, the $p$-defect of any spin character
$\chi\in\Irr(G)$ is
at least the weight of the corresponding $p$-block $B$. On the other hand,
by \cite[Prop.~5.2]{MR16} the blocks of $G$ satisfy the $l(B)$-conjecture,
so Conjecture~B holds. Now for any $p$-block of $2.\fA_n$ there exists a
height and defect group preserving bijection to a $p$-block of a suitable
$2.\fS_m$, so the claim for $2.\fA_n$ also follows.
\end{proof}

\subsection{}
We now turn to verifying our conjectures for the sporadic quasi-simple groups.

\begin{prop}   \label{prop:spor}
 Let $B$ be a $p$-block of a covering group of a sporadic simple group or of
 $\tw2F_4(2)'$. Then:
 \begin{enumerate}[\rm(a)]
  \item $B$ satisfies Conjecture~B, unless possibly when
   $B$ is as in the first four lines of Table~\ref{tab:spor}.
  \item $B$ satisfies Conjecture~A unless possibly when
   $B$ is as in the last four lines of Table~\ref{tab:spor}.
 \end{enumerate}
\end{prop}

\begin{table}[ht]
\caption{Blocks in sporadic groups}   \label{tab:spor}
  $$\begin{array}{c|cccrrc}
 G& p& d(B)& \max\Ht& l(B)& \chi(1)& \text{Conj.~A(1)}\\
 \hline
     Ly& 2& 8& 5& 9& 4\,997\,664& \text{ok}\\
   Co_1& 3& 9& 6& 29& 469\,945\,476\\
   Co_1& 5& 4& 2& 29& 210\,974\,400& \text{ok}\\
 2.Co_1& 5& 4& 2& 29& 1\,021\,620\,600\\
    J_4& 3& 3& 1& 9& \text{5 chars in 2 blks}\\
 \end{array}$$
\end{table}

\begin{proof}
For most blocks of sporadic groups, the inequality~(\ref{eq:height}) can be
checked using the
known character tables and Brauer tables; the only remaining cases are listed
in Table~\ref{tab:spor}, where the first four lines contain cases in which
Conjecture~B might fail, while the last four lines are those cases where
Conjecture~A might fail.
\par
In two of these remaining cases we can show that at least Conjecture~A(1) holds.
For $Ly$, the tensor product of the 2-defect~0 characters of degree 120064
with the irreducible character of degree~2480 is projective, has non-trivial
restriction to the principal block, but does not contain the (unique)
character $\chi$ of defect~2 of degree~4\,997\,664. Thus $|\IBr(\chi^0)|\le8$.
\par
For $Co_1$ the tensor products of irreducible 5-defect zero characters with
irreducible characters, restricted to the principal block, span a 7-dimensional
space of projective characters not containing the unique defect~2 character
$\chi$ of degree 210\,974\,400, so $|\IBr(\chi^0)|\le22$.
\end{proof}

\section{Groups of Lie type}

In this section we consider our conjectures for quasi-simple groups of Lie
type $G$. We prove both Conjectures~A and~B when $p$ is the defining
characteristic of $G$, up to finitely many symplectic groups. On the other
hand, we only treat one series of examples in the case of non-defining
characteristic.

\subsection{Defining characteristic}
We need an auxiliary result about root systems:

\begin{lem}   \label{lem:max N}
 Let $\Phi$ be an indecomposable root system and denote by $N(\Phi)$ its number
 of positive roots. If $\Psi\subset\Phi$ is any proper subsystem of $\Phi$
 then $N(\Phi)-N(\Psi)\ge n$, where $n$ is the rank of $\Phi$.
\end{lem}

In fact, our result is more precise in that we determine the minimum of
$N(\Phi)-N(\Psi)$ for each type, see Table~\ref{tab:NPhi}.

\begin{proof}
The values of $N(\Phi)$ for indecomposable root systems $\Phi$ are given as
in Table~\ref{tab:NPhi} (see e.g. \cite[Tab.~24.1]{MT}).
\begin{table}[h]   \label{tab:NPhi}
\caption{Maximal subsystems}
  $$\begin{array}{c|cccrrrrr}
 \Phi& A_n& B_n,C_n& D_n\ (n\ge4)& G_2& F_4& E_6& E_7& E_8\\
 \hline
 N(\Phi)& \binom{n+1}{2}& n^2& n^2-n& 6& 24& 36& 63& 120\\
 \text{max. }N(\Psi)& \binom{n}{2}& n^2-n& n^2-3n+2& 2& 16& 20& 36& 64\\
 N(\Phi)-N(\Psi)& n& n& 2n-2& 4& 8& 16& 27& 56\\
\end{array}$$
\end{table}
The possible proper subsystems can be determined by the algorithm of
Borel--de Siebenthal (see \cite[\S13.2]{MT}). For $\Phi$ of type $A_n$ the
largest proper subsystem $\Psi$ has type $A_{n-1}$, and then
$N(\Phi)-N(\Psi)=\binom{n+1}{2}-\binom{n}{2}$. For $\Phi$ of type $B_n$, we
need to consider subsystems of types $D_n$ and $B_{n-1}B_1$, of which the
second always has the larger number of positive roots. Next, for type $D_n$
with $n\ge4$, the largest subsystems are those of types $A_{n-1}$ and
$D_{n-1}$, which lead to the entries in our table. Finally, it is
straightforward to handle the possible subsystems for $\Phi$ of exceptional
type.
\end{proof}

\begin{cor}   \label{cor:defchar-lb}
 Let $G$ be a finite quasi-simple group of Lie type in characteristic~$p$.
 Let $\chi\in\Irr(G)$ lie in the $p$-block $B$. Then $l(B)\le|G_p|/\chi(1)_p$.
 In particular Conjecture~B holds for $G$ at the prime $p$.
\end{cor}

\begin{proof}
The faithful $p$-blocks of the finitely many exceptional covering groups can be
seen to satisfy~(\ref{eq:height}) by inspection using the Atlas \cite{Atl}.
Note that the non-exceptional
Schur multiplier of a simple group of Lie type has order prime to the
characteristic (see e.g.~\cite[Tab.~24.2]{MT}). Thus we
may assume that $G$ is the universal non-exceptional covering group of its
simple quotient $S$. Hence, $G$ can be obtained as the group of fixed
points $\bG^F$ of a simple simply connected linear algebraic group $\bG$ over
an algebraic closure of $\FF_p$ under a Steinberg endomorphism
$F:\bG\rightarrow\bG$. \par
Now the order formula \cite[Cor.~24.6]{MT} shows that $|G|_p=q^N$, where $q$
is the underlying power of $p$ defining $G$, and $N=N(\Phi)=|\Phi^+|$ denotes
the number of positive roots of the root system $\Phi$ of $\bG$. According to
Lusztig's Jordan decomposition of ordinary irreducible characters of $G$, any
$\chi\in\Irr(G)$ lies in some Lusztig series $\cE(G,s)$, for $s$ a semisimple
element in the dual group
$G^*$, and $\chi(1)=|G^*:C_{G^*}(s)|_{p'}\ \psi(1)$ for some unipotent
character $\psi$ of $C_{G^*}(s)$. In particular, $\chi(1)_p=\psi(1)_p$ is
at most the $p$-part in $|C_{G^*}(s)|$, hence at most $q^M$ for $M$ the
number of positive roots of the connected reductive group $C_{\bG^*}^\circ(s)$.
(Observe that $|C_{G^*}(s):C_{G^*}^\circ(s)|$ is prime to $p$ by
\cite[Prop.~14.20]{MT}.)
Now first assume that $s\ne1$, so $s$ is not central in $G^*$ (which is of
adjoint type by our assumption on $\bG$). Then the character(s) in $\cE(G,s)$
with maximal $p$-part correspond to the Steinberg character of
$C_{\bG^*}^\circ(s)$, of degree $q^{N(\Psi)}$, where $\Psi$ is the root
system of $C_{\bG^*}^\circ(s)$, a proper subsystem of $\Phi$. In particular,
$|G|_p/\chi(1)_p\ge q^n$ by Lemma~\ref{lem:max N}, with
$n$ denoting the Lie rank of $\bG$.
\par
We next deal with the Lusztig series of $s=1$, that is, the unipotent
characters of $G$. If $\chi\in\Irr(G)$ is unipotent then its degree is given by
a polynomial in $q$, $\chi(1)=q^{a_\chi}f_\chi(q)$, where $f_\chi(X)$ has
constant term $\pm1$, so that $\chi(1)_p=q^{a_\chi}$ (see
\cite[13.8 and 13.9]{Ca}). We let $A_\chi$ denote
the degree of the polynomial $X^{a_\chi}f_\chi(X)$. Let $D(\chi)$ denote the
Alvis--Curtis dual of $\chi$. Then $D(\chi)(1)=q^{N-A_\chi}f_\chi'(q)$, for
some polynomial $f_\chi'$ of the same degree as $f_\chi$, so the degree
polynomial of $D(\chi)$ is of degree $N-a_\chi$. We are interested in
unipotent characters $\chi$ with large $a_\chi$, that is, those for which
the degree polynomial of the Alvis--Curtis dual has small degree $N-a_\chi$.
The Alvis--Curtis dual of the trivial character is the Steinberg character,
whose degree is just the full $p$-power $q^N$ of $|G|$. The smallest possible
degrees of degree polynomials of non-trivial unipotent characters for simple
groups of Lie type are easily read off from the explicit formulas
in \cite[13.8 and~13.9]{Ca}, they are given as follows:
$$\begin{array}{c|cccccccccc}
 \Phi& A_n& B_n,C_n& D_n\ (n\ge4)& G_2& F_4& E_6& E_7& E_8\\
 \hline
  N-a_\chi& n& 2n-1& 2n-3& 5& 11& 11& 17& 29\\
\end{array}$$
It transpires that again $|G|_p/\chi(1)_p\ge q^n$ in all cases.
\par
On the other hand, the $p$-modular irreducibles of $G$ are parametrised by
$q$-restricted weights of $\bG$, so there are $q^n$ of them, one of which is
the Steinberg character, of defect zero. Thus $l(B)<q^n$ for all $p$-blocks $B$
of $G$ of positive defect which shows Conjecture~B.
\end{proof}

We now turn to Conjecture~A. In view of
Corollary~\ref{cor:defchar-lb} only its second assertion remains to be
considered. The following result shows that the assumptions of
Conjecture~A are hardly ever satisfied:

\begin{prop}   \label{prop:defchar}
 Let $G$ be a quasi-simple group of Lie type in characteristic $p$. Assume that
 $G$ has an irreducible character $\chi\in\Irr(G)$ such that
 $|G|_p=p^2\,\chi(1)_p$. Then $G$ is a central quotient of $\SL_2(p^2)$,
 $\SL_3(p)$, $\SU_3(p)$ or $\Sp_4(p)$.
\end{prop}

\begin{proof}
Again the finitely many exceptional covering groups can be handled by
inspection using the Atlas \cite{Atl} and so as in the proof of
Corollary~\ref{cor:defchar-lb} we may assume that $G$ is the universal
non-exceptional covering group of its simple quotient $S$ and so can be
obtained as the group of fixed points of a simply connected simple algebraic
group under a Steinberg map.
\par
Using Lusztig's Jordan decomposition we see that our question boils down to
\begin{enumerate}[\rm(1)]
 \item determining the irreducible root systems $\Phi$ possessing a proper root
  subsystem $\Psi$ such that $N(\Psi)\ge N(\Phi)-2$, and
 \item finding the unipotent characters $\chi$ of $G$ such that
  $p^2\chi(1)_p\ge|G|_p$.
\end{enumerate}
\par
The first issue can easily be answered using Lemma~\ref{lem:max N}. According
to Table~\ref{tab:NPhi} only $\bG$ of types $A_1$, $A_2$ or $B_2$ are
candidates, the first with $q=p^2$, that is, $G=\SL_2(p^2)$, and the other two
only for $q=p$. For type $A_2$ this leads to $\SL_3(p)$ and $\SU_3(p)$, in
the case of type $B_2$ we have $|\Phi(B_2)|-|\Phi(B_1^2)|=2$, which only leads
to $G=\Sp_4(p)$. Indeed, for the twisted Suzuki groups $\tw2B_2(q^2)$, where
$q^2$ is even, there is no centraliser of root system $A_1^2$ in the dual group.
\par
Issue~(2) has already been partly discussed in the proof of
Corollary~\ref{cor:defchar-lb}. Using the list of maximal $q$-powers occurring
in unipotent character degrees given there it follows that examples can only
possibly arise if the root system of $\bG$ is of type $A_1$ or $A_2$.
\end{proof}

\begin{thm}   \label{thm:defchar}
 Let $G$ be a quasi-simple group of Lie type in characteristic $p$. Then
 Conjecture~A holds for $G$ and the prime $p$, except possibly for $\Sp_4(p)$
 with $7\le p\le 61$.
\end{thm}

\begin{proof}
By Proposition~\ref{prop:defchar} we only need to consider the groups
$\SL_2(p^2)$, $\SL_3(p)$, $\SU_3(p)$ and $\Sp_4(p)$.
The ordinary character tables of all these groups are known and available
for example in the {\sf Chevie} system \cite{Chv}.
\par
The $p$-modular irreducibles of $G=\SL_2(p^2)$ are index by $p^2$-restricted
dominant weights, so there are exactly $p^2$ of them. One of them, the
Steinberg representation, is of defect~0, so any $p$-block of $G$ contains at
most $p^2-1$ irreducible Brauer characters. All decomposition numbers are equal
to~0 or~1 by Srinivasan \cite{Sr64}, which deals with this case.
\par
For $G=\SL_3(p)$ the only characters $\chi$ satisfying the assumptions of
Conjecture~A are the unipotent character $\chi_u$ of degree
$p(p+1)$ and the regular characters $\chi_r$ of degree $p(p^2+p+1)$.
The $p$-modular irreducibles of $G$ are parametrised by $p$-restricted weights,
so since $G$ is of rank~2 there are $p^2$ of them, one of which is the
Steinberg character, of defect zero. Thus $l(B)<p^2$ for all blocks of
positive defect which already shows~(1). Now assume that $p\ge11$.
By \cite[Tab.~1 and \S4]{Hu81}
$\chi_u^0$ has just two modular composition factors. Furthermore,
projective indecomposables of $G$ have degree at most $12p^3$ (see
\cite{Hu81}), so any decomposition number occurring for $\chi_r$ is at most
$12p^2/(p^2+p+1)$, which is smaller or equal to $p$ for $p\ge11$. The
decomposition numbers for $p\le7$ are contained in \cite{GAP}.
\par
The arguments for $G=\SU_3(p)$ are entirely similar, using the reference
\cite{Hu90} for decomposition numbers when $p\ge11$. The cases with
$3\le p\le7$ are contained in \cite{GAP}.
\par
Finally consider $G=\Sp_4(p)$. Here the only relevant characters $\chi$ are
those of degree $\frac{1}{2}p^2(p^2\pm1)$ parametrised by involutions in the
dual group with centraliser of type $A_1^2$; they only exist when $p$ is
odd, which we now assume. Then $G$ has two $p$-blocks of positive defect
both containing $(p^2-1)/2$ modular irreducibles. This already proves~(1).
The second claim follows as above using that the PIMs have degree at most
$32p^4$ (see \cite{Hu86}) whenever $32p^4\le p^3(p^2+1)/2$, so whenever
$p>61$. The Brauer tables of $\Sp_4(p)$ for $p=3,5$ are contained in \cite{GAP}.
\end{proof}

\subsection{Non-defining characteristic}
The current knowledge about decomposition numbers for blocks of groups of Lie
type in non-defining characteristic does not seem sufficient to prove even
Conjecture~A, not even for unipotent blocks in general. We hence just make
some preliminary observations.

\begin{exmp}   \label{exmp:GLn}
This example shows that characters in blocks of groups of Lie type in
non-defining characteristic can have rather large heights. Let $G=\GL_n(q)$,
and $\ell$ a prime dividing $q-1$. Then all unipotent characters lie in the
principal $\ell$-block $B_0$ of $G$, and so do all characters in Lusztig series
$\cE(G,s)$ for any $\ell$-element $t\in G^*\cong G$. Now assume that
$n=\ell^a$ for some $a\ge1$. Then $|G|_\ell=\ell^{cn+(n-1)/(\ell-1)}$ where
$\ell^c$ is the precise power of $\ell$ dividing $q-1$. Let $T\le\GL_n(q)$ be
a Coxeter torus, of order $q^n-1$, and $t\in T$ an element of maximal
$\ell$-power order. It can easily be seen that then
$o(t)=(q-1)_\ell\ell^a=\ell^{c+a}$ and $t$ is regular, so $\cE(G,t)$
consists of a single character $\chi$, say. Now $\chi(1)=|G:C_G(t)|$ and hence
$$\chi(1)_\ell=|G|_\ell/|T|_\ell=\ell^{cn+(n-1)/(l-1)-c-a},$$
so $\chi$ has height $c(n-1)+(n-1)/(l-1)-a$ and defect $c+a$. So the height
can become arbitrarily large by varying $q$ for fixed $n$ and in particular
it is not bounded in terms of the (relative) Weyl group. Moreover, the
unipotent characters form a basic set for $B_0$. As they are in bijection
with $\Irr(\fS_n)$, $l(B_0)$ is the number of partitions of $n$ and
so grows exponentially in $n$.   \par
Choosing $a=c=1$, that is, $n=\ell$ and $\ell||(q-1)$ we obtain principal
$\ell$-blocks of defect $\ell+1$ containing a character of defect~2.
These examples also give rise to similar blocks for the quasi-simple groups
$\SL_n(q)$. We will deal with them in Proposition~\ref{prop:SLn} below.
\end{exmp}

Now let $G$ be the group of fixed points of a connected reductive linear
algebraic group $\bG$ under a Frobenius endomorphism $F$ defining an
$\FF_q$-structure.
Let $\ell$ be a prime different from the defining characteristic of
$G$ and set $e=e_\ell(q)$, the order of $q$ modulo~$\ell$ if $\ell>2$,
respectively the order of $q$ modulo~4 if $\ell=2$. We write
$a_G(e)$ for the precise power of $\Phi_e$ dividing the order polynomial of
the derived subgroup $[\bG,\bG]$, that is, since $\bG=[\bG,\bG]Z(\bG)$,
the precise power of $\Phi_e$ dividing the order polynomial of $\bG/Z(\bG)$.
If $e$ is clear from the context, we will just denote it by $a_G$.

We start with an observation on defects of characters in unipotent
$e$-Harish-Chandra series of $G$.

\begin{lem}   \label{lem:a-value}
 Let $\bG,F,q$ be as above. Let $\chi$ be a unipotent character of $G=\bG^F$
 lying in the $e$-Harish-Chandra series of the $e$-cuspidal pair $(\bL,\la)$.
 Let $\ell$ be a prime with $e=e_\ell(q)$ and set $\ell^c||\Phi_e(q)$.
 If $\chi$ has defect~2, then $c(a_G-a_L)\le2$. Moreover, if $c(a_G-a_L)=2$
 then $\chi$ is $e\ell^a$-cuspidal for all $a>0$.
\end{lem}

\begin{proof}
As the unipotent character $\chi$ lies in the $e$-Harish-Chandra series of
$(\bL,\la)$, the $\Phi_e$-part of its degree polynomial agrees with
the one of $\la$ (see \cite{BMM}). As $\la$ is $e$-cuspidal, the
$\Phi_e$-part in its degree polynomial equals $a_L$ \cite[Prop.~2.4]{BMM}. So
$|G|/\chi(1)$ is divisible by at least $\Phi_e^{a_G-a_L}$ and hence by
$\ell^{c(a_G-a_L)}$. If $\chi$ has defect~2 this implies that $c(a_G-a_L)\le2$.
If $c(a_G-a_L)=2$ then, since $\ell|\Phi_{e\ell^a}$, $\chi(1)$ must be
divisible by the same power of $\Phi_{e\ell^a}$ as $|G|$ for all $a>0$, that
is, $\chi$ must be $e\ell^a$-cuspidal.
\end{proof}

Now assume that $\ell\ge5$ is a prime that is good for $\bG$. By the main
result of Cabanes--Enguehard \cite{CE94} the
unipotent $\ell$-blocks of $G$ are then parametrised by $e$-cuspidal unipotent
pairs $(\bL,\la)$ up to conjugation. Let $B=B_G(\bL,\la)$ be a unipotent
$\ell$-block of $G$. Then according to \cite[Thm.]{CE94} the irreducible
characters in $B$ are the constituents of $R_{G_t}^G(\hat t\chi_t)$ where $t$
runs over $\ell$-elements in $G^*$, $\bG_t$ is a Levi subgroup of $\bG$ dual
to $C_{\bG^*}^\circ(t)$, $\hat t$ is the linear character of $G_t$ dual to $t$,
and $\chi_t$ lies in the unipotent $\ell$-block $B_{G_t}(\bL_t,\la_t)$, where
$(\bL_t,\la_t)$ is a unipotent $e$-cuspidal pair of $\bG_t$ such that
$[\bL,\bL]=[\bL_t,\bL_t]$ and $\la,\la_t$ have the same restriction to
$[\bL,\bL]^F$. Now as $\bG_t^*=C_{\bG^*}(t)$, $R_{G_t}^G$ induces a bijection
(with signs) $\cE(G_t,\hat t)\longrightarrow\cE(G,\hat t)$ such that degrees
are multiplied by $|G:G_t|_{p'}$. In particular this shows that $\chi$ and
$\chi_t$ have the same $\ell$-defect.

\begin{prop}   \label{prop:smallcent}
 Let $\bG,F,q,\ell$ be as above. Assume that the principal $\ell$-block of
 $G$ contains a character of $\ell$-defect~2. Then a Sylow $\ell$-subgroup
 $S$ of $G^*$ contains an element $t$ with centraliser $|C_S(t)|=\ell^2$.
\end{prop}

\begin{proof}
The principal $\ell$-block $B_0$ of $G$ is the $\ell$-block above the
$e$-cuspidal pair $(\bL,1_L)$, where $\bL$ is the centraliser of a Sylow
$e$-torus of $\bG$. In particular $a_L=0$. Now let $\chi\in\Irr(B_0)$ of
defect~2.
By the result of Cabanes and Enguehard cited above, there is an $\ell$-element
$t\in G^*$ and a unipotent character $\chi_t\in\Irr(G_t)$ of defect~2 in the
$e$-Harish-Chandra series of $(\bL_t,1_{L_t})$ with $[\bL_t,\bL_t]=[\bL,\bL]$.
By Lemma~\ref{lem:a-value} this implies that $ca_{G_t}\le2$. In particular
$\Phi_e$ divides the order polynomial of $[\bG_t,\bG_t]$ at most twice, and
so the same statement holds for the dual group $\bC:=C_{\bG^*}(t)$. An
inspection of the order formulas of the finite reductive groups shows that
then the Sylow $\ell$-subgroups of $C$ are abelian (using that $\ell>3$), and
hence contained in a Sylow $e$-torus of $\bC$. For $\chi_t$ to have defect~2
this forces $|C|_\ell=\ell^2$, so $t$ is as claimed.
\end{proof}

We will not attempt to classify the cases when Sylow $\ell$-subgroups of
finite reductive groups contain $\ell$-elements with this property, even though
that seems possible, since in most of these cases our knowledge on
decomposition numbers would not suffice to settle Conjecture~A anyway. We just
discuss one particular case.

\begin{lem}   \label{lem:centPGL}
 Let $G=\PGL_n(q)$ with $n\ge4$ and $\ell$ a prime dividing $q-1$. If $G$
 contains an $\ell$-element $t$ with $|C_G(t)|_\ell=\ell^2$ then one of
 \begin{enumerate}[\rm(1)]
  \item $n=\ell$ and $C\cong\GL_1(q^\ell)$,
  \item $n=\ell+1$ and $C\cong\GL_\ell(q)\times\GL_1(q)$, or
  \item $n=\ell^2$ and $C\cong\GL_1(q^{\ell^2})$;
 \end{enumerate}
 where $C$ denotes the centraliser in $\GL_n(q)$ of a preimage of $t$ under the
 natural map.
\end{lem}

\begin{proof}
Let $\hat t$ be a preimage of $t$ of $\ell$-power order in $\hat G:=\GL_n(q)$
under the natural surjection. Then $|C_{\hat G}(\hat t)|_\ell\le\ell^{2+c}$
where $\ell^c$ is the precise power of $\ell$ dividing $q-1$. Now
$$C:=C_{\hat G}(\hat t)\cong
  \GL_{n_1}(q^{a_1})\times\cdots\times\GL_{n_r}(q^{a_r})$$
for suitable $n_i,a_i\ge1$ with $\sum n_ia_i=n$. Let $s_i$ denote the projection
of $\hat s$ into the $i$th factor of $C$. Then $s_i\in
Z(\GL_{n_i}(q^{a_i}))\cong\FF_{q^{a_i}}^\times$ is an $\ell$-element generating
the field $\FF_{q^{a_i}}$. In particular, $\ell$ divides the cyclotomic
polynomial $\Phi_{a_i}(q)$, and so $a_i=\ell^{f_i}$ for some $f_i\ge0$ (see
e.g.~\cite[Lemma~25.13]{MT}). Then $|C|_\ell\ge \sum_{i=1}^r (c+f_i)n_i$.
Our assumption thus implies that $\sum n_i\le3$.
\par
If $n_1=3$ then $c=1$, $f_1=0$, so $a_1=1$ and $n=\sum n_i=3$, which was
excluded. Now assume that $\sum n_i=2$. If $n_1=2$ then $C\cong\GL_2(q^{n/2})$
and $f_1=0$, whence $n=2$. If $n_1=n_2=1$ then
$C\cong\GL_1(q^{a_1})\times \GL_1(q^{a_2})$ and $2c+f_1+f_2\le2+c$, so $c=1$
and $f_1+f_2\le1$. This leads to $C\cong\GL_1(q)^2\le\GL_2(q)$ or to
$C\cong\GL_\ell(q)\times\GL_1(q)\le\GL_{\ell+1}(q)$, with $\ell||(q-1)$.
Finally assume that $\sum n_i=1$, so $n_1=1$ and
$C\cong\GL_1(q^n)$ with $n=\ell^f$ and $c+f\le3$. If $c=3$ then $f=0$ and
$n=1$; if $c=2$ then we find $C\cong\GL_1(q^\ell)\le\GL_\ell(q)$, and if
$c=1$ then we also could have $C\cong\GL_1(q^{\ell^2})\le\GL_{\ell^2}(q)$.
\end{proof}

\begin{prop}   \label{prop:SLn}
 Conjecture~A holds for the principal $\ell$-block of $\SL_n(q)$ when
 $\ell|(q-1)$.
\end{prop}

\begin{proof}
Assume that the principal $\ell$-block $B_0$ of $\SL_n(q)$ contains a character
$\chi\in\Irr(B_0)$ of $\ell$-defect~2. Then $\chi\in\cE(G,t)$ for some
$\ell$-element $t\in G^*=\PGL_n(q)$ such that $|C_{G^*}(t)|_\ell=\ell^2$ by
Proposition~\ref{prop:smallcent}. According to Lemma~\ref{lem:centPGL} then
either $n\le3$ or $n\in\{\ell,\ell+1,\ell^2\}$. For $n\le3$ the claim is easily
checked from the known decomposition matrices. We consider the remaining
cases in turn.
\par
Assume that $n=\ell^f$ with $f\in\{1,2\}$. Then $C_{\bG^*}(t)$ is a Coxeter
torus of $\PGL_n$ by Lemma~\ref{lem:centPGL}(1) and~(3), and $t$ is a regular
element. Its centraliser is disconnected, with
$|C_{G^*}(t)|=\ell^f(q^{\ell^f}-1)/(q-1)$, so the characters in these Lusztig
series can only have defect~2 if $f=1$, that is, $n=\ell$, which we assume
from now on. The $\ell$ characters in $\cE(G,t)$ are the constituents of the
restriction to $G$ of the unique character $\psi$ in $\cE(\tG,\hat t)$,
where $\tG=\GL_n(q)$. Let $\tT$ be a Coxeter torus of $\tG$. Then
$\psi=\RtTtG(\theta)$ for an $\ell$-element $\theta\in\Irr(\tT)$ in duality
with $\hat t$. The reduction of $\RtTtG(\theta)$ modulo~$\ell$ thus coincides
with the reduction of $\RtTtG(1)$. The latter decomposes as
$\sum_{i=1}^\ell(-1)^{i-1}\chi_i$, where $\chi_i\in\Irr(B_0)$ is the unipotent
character of $\tG$ parametrized by the hook partition $(i,1^{\ell-i})$.
Now the decomposition numbers of the unipotent characters in $B_0$ are known:
the $\ell$-decomposition matrix of the Iwahori--Hecke algebra of type
$A_{\ell-1}$, that is, of the symmetric group $\fS_\ell$ embeds into that
of $G$. Since $\ell$ divides $|\fS_\ell|$ just once, we obtain a Brauer tree
with $\chi_i+\chi_{i+1}$ being projective for $1=1,\ldots,\ell-1$. Adding up
we see that $\psi$ is irreducible modulo~$\ell$, hence the same is true for
$\chi\in\cE(G,t)$ and our claim is proved in this case.
\par
Assume next that $n=\ell+1$. Then by Lemma~\ref{lem:centPGL}(2) the centraliser
of $t$ is a maximal torus with $|C_{G^*}(t)|=q^\ell-1$. So again $t$ is
regular, but now with connected centraliser, hence $\cE(G,t)=\{\chi\}$ consists
of just one character. Again the principal block of $\fS_{\ell+1}$ has cyclic
defect, and a computation as in the previous case shows that $\chi$ is in
fact irreducible. The validity of Conjecture~A follows.
\end{proof}


\end{document}